\documentclass[a4paper]{article}
\usepackage[affil-it]{authblk}

\usepackage{amsthm, amsmath, amssymb}
\usepackage{comment}
\usepackage{color}
\usepackage{mathtools}

\providecommand{\keywords}[1]
{
	\small	
	\textbf{\textit{Keywords---}} #1
}

\providecommand{\subjclass}[1]
{
	\small	
	\textbf{\textit{AMS Subject Classification.}} #1
}

\newtheorem{theorem}{Theorem}
\newtheorem{lemma}{Lemma}

\theoremstyle{definition}

\theoremstyle{remark}
\newtheorem{remark}{Remark}

\newcommand{\cL}{{\mathcal L}}

\date{}

\begin{document}

\title{ \textbf {Uniqueness of a solution to a general class of  discrete system defined on connected graphs}}

\author[1]{Avetik Arakelyan
	\thanks{email: \texttt{arakelyanavetik@gmail.com}}
}
\author[2]{Farid Bozorgnia
	\thanks{email: \texttt{farid.bozorgnia@tecnico.ulisboa.pt}}
}

\affil[1]{Institute of Mathematics, NAS of Armenia, 0019 Yerevan, Armenia}
\affil[1]{TenWeb IO, LLC, Gyulbenkyan 30/3, apt 141, 0033 Yerevan, Armenia}
\affil[2]{CAMGSD, Dept. of Mathematics, Instituto Superior T\'{e}cnico, Lisbon, Portugal}


\maketitle

\begin{abstract}
	In this work we prove uniqueness result for an implicit discrete system defined on connected graphs.
	Our discrete system is motivated from a certain class of spatial segregation of reaction-diffusion equations.
	
\end{abstract}
\[\]
\\


\section{Introduction}

In recent years there have been intense studies of spatial segregation for reaction-diffusion systems. The existence and uniqueness of spatially inhomogeneous solutions for competition models of Lotka-Volterra type in the case of two and more competing densities  have been investigated  in
\cite{Av-Farid-unique, MR2146353,MR2151234, MR2300320, MR1900331, MR1459414, MR2417905}.
The aforementioned segregation problems led to an interesting class of  multi-phase free boundary problems. These   problems have growing interest due to their  important applications in the different branches of applied mathematics. To see the diversity of applications we refer \cite{Avet-Henrik,bucur-multi,Aram-Caff} and the references therein.

In a recent series of papers the authors consider numerical approach to a certain class of spatial segregation problem (see \cite{Av-conv, Av-Raf-2016, MR2563520, Mywork}). It was developed certain finite difference scheme and proved its solution existence, uniqueness and convergence. In the present work the author concerns to generalize the developed difference scheme for more complicated domains, namely for connected graphs. In short, our work is devoted to generalize and prove uniqueness of competing densities configuration over connected graphs.

Our motivation for current work is    segregation model and Dirichlet partition problem  which we will briefly explain in the following section for both continuum and discrete cases.  Nonlocal PDEs on graphs have recently received a lot of attention due to their applications in the real world. Assume we are give a data set and the aim   is to  identifying meaningful  clusters.  One way to  perform  the
clustering   is to construct a weighted graph, $G = (V,E,W),$   where the vertices $V$ represent
the data   to be clustered, $E$ is a set of edges and the weights $W$ indicates the  similarity between these data  and is defined  on the
edges.  To partition the resulting graph,   different  methods have been investigated  e.g., spectral clustering
\cite{Lux}, Dirichlet partitioning \cite{Zos},   minimizing graph cuts \cite{Gar}. Note that the continuum and discrete problems should be consistent means if both  problems of interests are stated in the variational form then the minimizer of discrete problem converge (in the appropriate sense) to the
minimizer of continuous  problem in the large sample limit as $ n$ tends to infinity.

Let  $G=(V,E)$ be a given connected graph. The aim of this work is to  prove uniqueness of the solution of following general system (we provide suitable assumptions on   $H$ and $f_{l}$ in Subsection $1.4$)
\begin{equation}\label{scheme_sys_intr}
\begin{cases}
u^{l}(x) =\max \left(H\left(x, \overline{u}^l(x) - \sum\limits_{p \neq l}  \overline{u}^p(x)\right)-f_{l}(x,u^{l}(x)) , \,  0\right),\;\;x\in G^o,\\
u^{l}(x) =\phi^{l}(x),\;\;x\in \partial G.
\end{cases}
\end{equation}
for every $l=1,2,\dots,m.$ Here $\overline{u}^l(x)$ denotes the average of  ${u}^l$ at  vertices $x$ (the precise definition will be given in Subsection $1.3$).

\subsection{Mathematical background}

The first model  is  related to asymptotic behaviour  of  the following coupled system as parameter $ \varepsilon$ tends to zero.

 \begin{align}\label{f20}
\begin{cases}
\Delta  u_{i}^{\varepsilon}=   \frac{ u_{i}^{\varepsilon}}{\varepsilon}  \sum\limits_{j \neq i}   u_{j}^{\varepsilon} (x)\qquad\qquad & \text{ in  } \Omega,\\
u_{i}^{\varepsilon} \ge 0\;  & \text{ in  } \Omega,\\
u^{\varepsilon}_{i}(x) =\phi_{i}(x)    &   \text{ on} \,  \partial \Omega,\\
  \end{cases}
\end{align}
for  $i=1,\cdots, m.$   The boundary values  satisfy
\[
\phi_{i}(x) \cdot \phi_{j}(x)=0,  \quad i \neq j \textrm{ on the boundary}.
\]
First, for each fixed  $\varepsilon $ the exist   unique positive solution  $ (u_{1}^{\varepsilon},\cdots   ,u_{m}^{\varepsilon})$. Next  by construction barrier functions, one can show that the normal derivative of $ u_{i}^{\varepsilon}$  is bounded independent of  $ \varepsilon$, this consequently proves that the  $H^1$-norm of  $ u_{i}^{\varepsilon}$ is bounded.   Let $ (u_1, \cdots  ,u_m)$ be the limiting configuration, then $u_i$ are pairwise segregated, i.e., $u_{i}(x)\cdot u_{j}(x)=0,$  each $u_i$   is harmonic in their supports and satisfy the following differential inequalities \cite{Aram-Caff,MR2146353}

 \begin{itemize}
\item  $ -\Delta u_{i} \ge 0$,\\
\item  $ -\Delta (u_{i}- \sum\limits_{j\neq i}u_j )\le 0$,
 \end{itemize}
 It has been shown \cite{MR2151234} that the limiting solution   $ (u_1, \cdots  ,u_m)$      minimizes the following functional
\begin{equation}\label{continuous-model}
\left\{
\begin{split}
&\min J(u):=\frac{1}{2}\int_\Omega\sum\limits_{i=1}^m|\nabla u_i|^2\, dx\\
&\text{subject to } u_i=\phi_i,\quad  \text{ on } \partial\Omega, \\[10pt]
&\qquad u_i\geq0, \quad\text{ and }\quad u_i\cdot u_j=0\quad\text{ in }\Omega,
\end{split}\right.
\end{equation}

 In  \cite{Mywork,Av-Raf-2016}, the authors have proposed the following iterative  scheme to solving \eqref{continuous-model} based on these properties:
 \begin{equation}\label{iteration}
u_{i}^{t+1} (x) =\max\Big(\overline{u}^{t}_{i}(x)-\sum\limits_{j\neq i}\overline{u}^{t}_{j}(x),\,  0\Big)\quad i=1\cdots m,
\end{equation}
where  $\overline{u}(x)$ denotes the average of values of  $u$  for neighboring points of the  mesh point $x,$  and $t$ refers to iterations.

 We would like to mention  that the scheme (\ref{iteration}) is implemented  to develop a new graph based semi-supervised algorithm for classification (see \cite{semi-superviser}).

 Our second example     is optimal partition  of first eigenvalue of Laplace operator which has application in data analysis. Given a bounded open set
$\Omega  \subset  \mathbb{R}^d,$  an  $m$-partitions of $\Omega$ is   a family of pairwise  disjoint, open and connected  subsets   ${\{\Omega_{i}}\}_{i=1}^{m}$
 such that
 \[
\overline{\Omega}_{1} \cup\overline{\Omega}_{2} \cup \cdots  \cup \overline{\Omega}_{m} =  \overline{\Omega}, \, \Omega_{i}\cap \Omega_{j}=\emptyset \quad \text{for } i\neq j.
 \]
 By $\mathfrak{D}_{m}$ we mean the set of all $m$-partitions of $\Omega$.    For  an   arbitrary partition
 $\mathfrak{D}= (\Omega_{1},\cdots ,\Omega_{m}) \in \mathfrak{D}_{m},$  we consider  the  following  minimization of domain functional
\begin{equation}\label{1}
\mathfrak{L_{m}} (\Omega)=\underset{\mathfrak{D} \in \mathfrak{D_m}}{\inf}  \frac{1}{m}  \sum_{i=1}^{m} \lambda_{1}(\Omega_{i}).
\end{equation}
Such $m$-partitions of $\Omega$ for which the minimum is realized are   called  an  optimal partitions. The first eigenvalue has a variational characterization,
 so  the  minimization problem in (\ref{1})  can be written  in term of  energies  as
    \begin{equation}\label{2}
\text{Minimize  }  E(u_1, \cdots, u_m)=\int_{\Omega}  \sum_{i=1}^{m}  \frac{| \nabla u_{i}|^{2}}{2}  \,  dx,
\end{equation}
  over the set
  $${\{(u_1,\dots,u_{n})\in (H^{1}_{0}    (\Omega))^{m } :  \|u_{i} \|_{L^{2}(\Omega)}=1, \,  u_{i} \cdot u_{j}=0, \text{ for } i\neq j }\}.$$
In  this case the optimal partitions  are  the supports of $u_{i},$ i.e.,
$$\Omega_{i}={\{x \in \Omega: u_{i}(x)> 0}\}.$$

Existence of optimal partitions for Problem  (\ref{1})  in the class of quasi-open sets  has been studied in
\cite{66}.  In
  \cite{77}    Conti, Terracini, and Verzini  studied regularity and qualitative properties  of the optimal
 partitions.    A numerical method to    approximate  partitions of a domain minimizing the sum of Dirichlet
 Laplacian eigenvalues   is given in \cite{farid-optimal-partition}, where  the author   consider   the system of differential inequalities.
 Using this result, they build   a numerical algorithm to approximate optimal
configurations.

     To penalize  the restriction $ u_{i} \cdot u_{j}=0$ we add the p term  $\frac{1}{\epsilon} u_{i}^2 u_{j}^{2}$ to the functional
with  $\epsilon$ tends to zero. Thus  we rewrite (\ref{2}) as
 \begin{equation}\label{722}
\text{  Minimize  } \int_{\Omega}  \sum_{i=1}^{n}  \frac{| \nabla u_{i}|^{2}}{2}+  \sum_{ j\neq i } \frac{1}{\epsilon} u_{i}^2 u_{j}^{2}    \,  dx,
\end{equation}
  over the set
  $${\{(u_1,\dots,u_{n})\in (H^{1}_{0}    (\Omega))^{n } :\int_{\Omega} u_{i}^{2} \, dx=1 }\}.$$
The minimizer $(u_1, \cdots, u_n)$ of (\ref{722}) satisfies the following  system
\begin{equation}\label{333}
\left \{
\begin{array}{lll}
 - \Delta u_{i}=  \lambda_{1}(\Omega_{i}) u_{i}- \frac{u_{i}}{\epsilon}   \sum_{ j\neq i }   u_{j}^{2}    &  \text{in }\,  \Omega \\
 u_{i} =0        & \text{on }    \partial  \Omega\\
\|u_{i}\|_{L^{2}}=1  \quad\\
 i=1,\cdots,n .
  \end{array}
\right.
\end{equation}

In  discrete case,  the Dirichlet  partition
problem is as following.  As the setting before let   $ G = (V, W)$   be  a weighted graph.      The weighted Dirichlet energy of a function  $ u: V\rightarrow \mathbb{R}$ is  given by

\begin{equation}\label{objective}
 I_{n}(u):=\frac{1}{2}\sum\limits_{i=1}^{n}\sum\limits_{j=1}^{n} w_{ij} (u(x_{i}) - u(x_{j}))^{2}.\\
\end{equation}

For any  subset   $ S\subset V$,  the first eigenvalue is defined as
\[
\lambda_{1}(S)= \underset{v|_{S^{c}}=0, \|v\|_{2}=1}{\textrm{min}} I(v)
\]
 For fixed $m$, the discrete Dirichlet $m$-partition problem on $G$ is then to choose disjoint subsets
$V_{1}, V_{2}, \cdots  V_m $   minimizing
\[
 \sum_{i=1}^{m} \lambda_{1}(V_{i}).
\]

\subsection{Connected graphs}
In this section for the sake of readers convenience we give a brief introduction to connected graphs.
A graph (sometimes called undirected graph for distinguishing from a directed graph) is a pair $G = (V, E),$ where $V$ is a set whose elements are called vertices (singular: vertex), and $E$ is a set of paired vertices, whose elements are called edges (sometimes links or lines).

The vertices $x$ and $y$ of an edge $(x, y)\in E$ are called the endpoints of the edge. The edge is said to join $x$ and $y$ and to be incident on $x$ and $y.$ A vertex may belong to no edge, in which case it is not joined to any other vertex.

In an undirected graph $G,$ two vertices $x$ and $y$ are called connected if $G$ contains a path from $x$ to $y.$ Otherwise, they are called disconnected. If the two vertices are additionally connected by a path of length 1, i.e. by a single edge, the vertices are called adjacent.

A graph is said to be \textbf{connected} if every pair of vertices in the graph is connected. This means that there is a path between every pair of vertices. An undirected graph that is not connected is called disconnected. An undirected graph $G$ is therefore disconnected if there exist two vertices in $G$ such that no path in $G$ has these vertices as endpoints. A graph with just one vertex is connected. An edgeless graph with two or more vertices is disconnected. For more detailed study we refer to \cite{Loebl2010} and references therein.

\subsection{Notations}
In this section we give  notations for the further exposition of the paper. First, we introduce the boundary notion for the connected graphs following the recent work by Steinerberger \cite{graph-boundary}. The author gives a definition, which coincides with what one would expected for the discretization of
(sufficiently nice) Euclidean domains. He also shows an isoperimetric inequality similar to the ones for regular Euclidean domains. Following \cite{graph-boundary}  we set the boundary of a given connected graph $G$ as follows:
\begin{equation}\label{boundary_def}
\partial G =\left\{x\in V\; \Big|\; \exists y\in V : \frac{1}{\deg(x)}\sum_{(x,z)\in E}d(z,y) < d(x,y)    \right\},
\end{equation}
where $G = (V,E),$ $d(x,y)$ is a distance between vertices $x$ and $y.$ By $\deg(x)$ we set the number of neighbor vertices close to $x.$ Note that $\deg(x) = |\{y\in V \;:\; (x,y)\in E\}|.$ Thus, we can easily define the notion of interior points of a graph $G$:
$$
G^o\equiv V\setminus\partial G.
$$
We again follow \cite{graph-boundary} to depict  in Figure $1$  different connected graphs and their boundaries according to a definition given in \eqref{boundary_def}.
\begin{remark}
We point out that there are other definitions of a boundary notion for graphs (see for instance \cite{chartrand-boundary}), but we consider the above definition due to its natural generalization of regular Euclidean domains boundary properties.
\end{remark}

\begin{figure}\label{figure}
	\centering
	\includegraphics[scale=0.17]{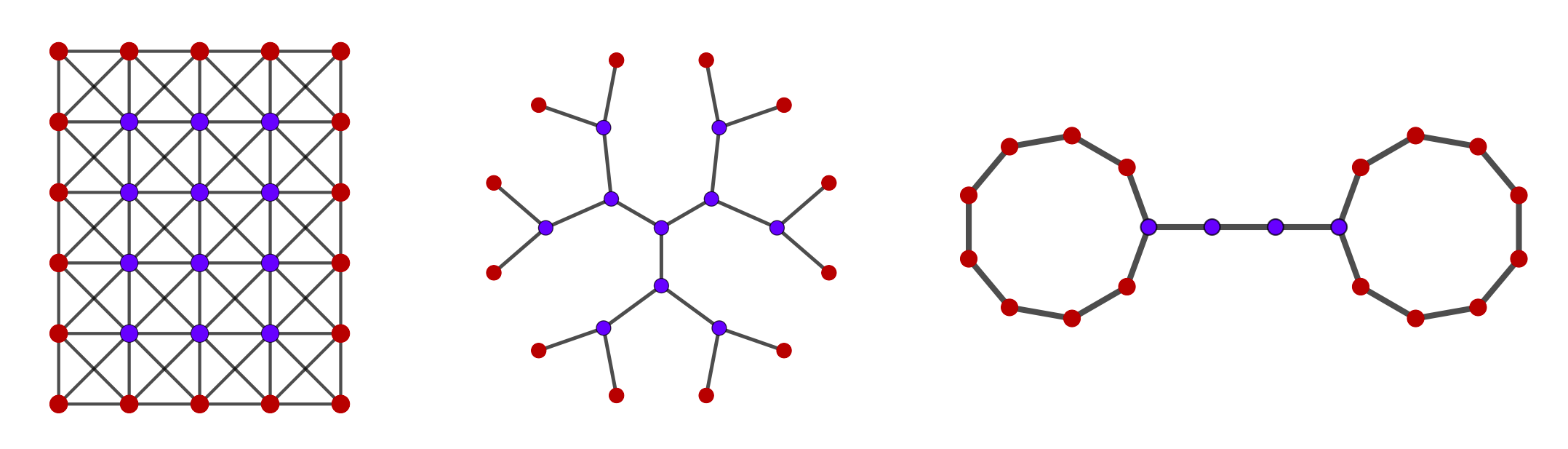}
	\includegraphics[scale=0.14]{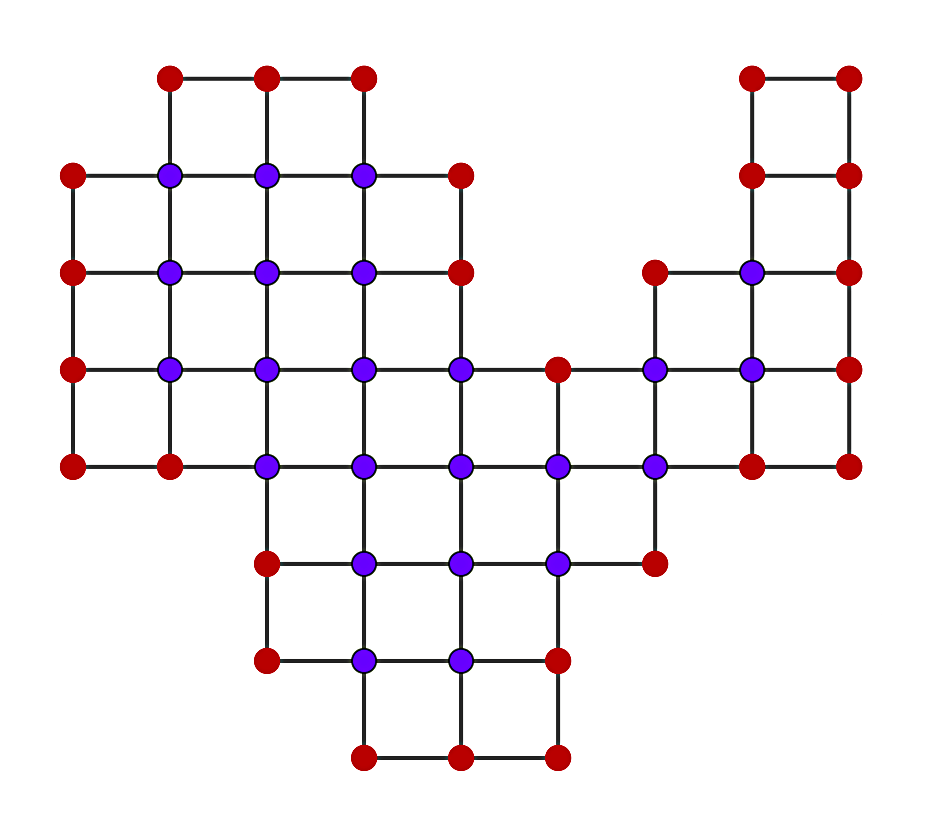}
	\caption{Graphs, their boundary $\partial G$ (red) and interior $V\setminus\partial G$ (blue).}	
\end{figure}
We define for a given function $r(x)$ the mean value around $x\in V$ as follows:
$$
\overline{r}(x) := \frac{1}{\deg(x)}\sum_{(x,y)\in E}r(y).
$$

We set by $u^l(x)$ the densities defined on $G,$ also we define  the functions $\phi^l(x),$ which are defined on $\partial G$ and assume they are  extended to be zero  on $G^o,$ for all $l=1,2,\dots, m$. We also impose disjointness condition on  $\phi^l(x),$ namely $\phi^i(x)\cdot\phi^j(x)=0,$ for every $i\neq j,$ and $x\in \partial G.$

According to the above defined densities we will need the following notation as well:
\[
\hat{u}^t(x):= {u}^t(x) - \sum\limits_{p \neq t}{u}^p(x).
\]

\subsection{The  setting of the problem and main result }
Assume $G=(V,E).$
Let $H(x,s):V\times\mathbb{R}^+\to\mathbb{R}$ and $f_i(x,s):V\times\mathbb{R}^+\to\mathbb{R}$ are continuous in $s,$ and $H(x,0)\equiv0, f_i(x,0)\equiv0.$
We focus on the following system:
\begin{equation}\label{scheme_sys}
\begin{cases}
u^{1}(x) =\max \left( H\left(x, \overline{u}^1(x) - \sum\limits_{p \neq 1}  \overline{u}^p(x)\right)-f_{1}(x,u^{1}(x)) , \,  0\right),\;\;x\in G^o,\\
u^{2}(x) =\max \left(H\left(x, \overline{u}^2(x) - \sum\limits_{p \neq 2}  \overline{u}^p(x)\right)-f_{2}(x,u^{2}(x)) , \,  0\right),\;\;x\in G^o,\\
\dots\dots\dots\dots\\
u^{m}(x) =\max \left(H\left(x, \overline{u}^m(x) - \sum\limits_{p \neq m}  \overline{u}^p(x)\right)-f_{m}(x,u^{m}(x)) , \,  0\right),\;\;x\in G^o,\\
u^{l}(x) =\phi^{l}(x),\;\;x\in \partial G.
\end{cases}
\end{equation}
for every $l=1,2,\dots,m.$
\begin{remark}
	Functions $f_i$'s are defined only for non negative values of s (recall that  densities $u_i$'s are non negative due to the above system \eqref{scheme_sys}); thus we can arbitrarily define such functions on the negative semiaxis. For the sake of convenience, when $s\leq 0$,  we will let $f_i(x,s)=-f_i(x,-s)$. This extension  preserves the continuity due to the conditions on $f_i$ defined above. In the same way, $H$ is extended on the negative semiaxis as well, i.e $H(x,s)=-H(x,-s).$
\end{remark}

As special case  of (\ref{scheme_sys}) let $m=1$ and $f(x)$ be non negative  function defined on $V$.  The one phase obstacle problem on  graph reads as
\begin{equation}\label{Onephase}
\left \{
\begin{array}{ll}
\min(-\cL u(x) + f(x), u)  =0,  &  x\in G^o,\\
  u=g,   &   \text{ on } \partial G,\\
  \end{array}
\right.
\end{equation}
 where $\cL$ is the unnormalized graph Laplacian given by
  \begin{equation*}
 \mathcal{L} u(x_i)=\sum\limits_{j=1}^{n} \, (u(x_{j}) - u(x_i)).
 \end{equation*}
 To solve (\ref{Onephase}) one can rewrite  the obstacle problem in the following form
 \[
 u(x)=\max \left( \overline{u}(x)- \frac{1}{\deg(x)} f(x), 0 \right).
 \]

The main result of the paper reads as follows:
\begin{theorem}\label{scheme_unique}
	Let the functions $f_{l}(x,s)$ and $H(x,s)$ be nondecreasing with respect to the variable $s$. Assume also that $H$ is a Lipschitz continuous with respect to $s$ and Lipschitz constant is $1,$ i.e. $|H(\cdot,s)-H(\cdot,t)|\leq |s-t|, \forall s,t\in\mathbb{R}^+.$ Then
	there exists  a unique vector  $(u^{1},u^{2},\dots,u^{m}),$ which
  satisfies the discrete system \eqref{scheme_sys}.
\end{theorem}

\section{Auxiliary lemmas}

We start this section by proving that the disjointness condition on $\partial G$ implies the same property for the densities $u^l(x)$ over whole graph $G.$ Let $G=(V,E),$ where $V$ is a set of vertices and $E$ is a set of edges to a given graph $G.$
\begin{lemma}\label{lemma_disj}
	Let the functions $f_{l}(x,s)$ and $H(x,s)$ be nondecreasing with respect to the variable $s$.   If a vector $(u^{1},u^{2},\dots,u^{m})$ solves the discrete system \eqref{scheme_sys}, then the following property holds:
	\[
	u^i(x)\cdot u^j(x)= 0,\; \forall x\in V,\; i\neq j.
	\]
\end{lemma}

\begin{proof}
	Let $u^i(x_0) > 0,$ for some $x_0\in G^o.$ Then
	$$
	u^i(x_0) =  H\left(x_0, \overline{u}^i(x_0) - \sum\limits_{p \neq i}  \overline{u}^p(x_0)\right)-f_{x_0}(x,u^{i}(x_0)) > 0.
	$$
	Therefore $ H\left(x_0, \overline{u}^i(x_0) - \sum\limits_{p \neq i}  \overline{u}^p(x_0)\right) > 0.$  This along with non-decreasing condition on $H$ and $H(x,0)=0,$ implies  $\overline{u}^i(x_0) - \sum\limits_{p \neq i}  \overline{u}^p(x_0)\geq0.$ Thus, for every $j\neq i$ we get $\overline{u}^j(x_0) - \sum\limits_{p \neq j}  \overline{u}^p(x_0)\leq 0.$
	This apparently yields
	$$
	H\left(x_0, \overline{u}^j(x_0) - \sum\limits_{p \neq j}  \overline{u}^p(x_0)\right)\leq H(x_0,0)=0,
	$$
	hence
	$$
	u^{j}(x_0) =\max \left( H\left(x_0, \overline{u}^j(x_0) - \sum\limits_{p \neq j}  \overline{u}^p(x_0)\right)-f_{j}(x_0,u^{j}(x_0)) , \,  0\right) = 0.
	$$
	This completes the proof.
	
\end{proof}

\begin{lemma}\label{lemma_ineq}
	Let the functions $f_{l}(x,s)$ and $H(x,s)$ be nondecreasing with respect to the variable $s$.   If  a vector $(u^{1},u^{2},\dots,u^{m})$ solves the discrete system \eqref{scheme_sys}, then the following inequality holds:
	\[
	H\left(x, \overline{u}^l(x) - \sum\limits_{p \neq l}  \overline{u}^p(x)\right)- \hat{u}^l(x)\leq f_{l}(x,u^{l}(x)),
	\]
	for all $l=1,2,...,m$ and $x\in V.$
\end{lemma}

\begin{proof}
	Let $u^l(x)$ is a fixed density. Assume for some $x_0\in G^o$ we have $u^l(x_0) > 0.$ Then according to Lemma \ref{lemma_disj} for every $p\neq l$ we obtain $u^p(x_0)=0.$ Hence $\hat{u}^l(x_0)={u}^l(x_0)$ and this in turn gives
	\begin{align*}
		H\left(x_0, \overline{u}^l(x_0) - \sum\limits_{p \neq l}  \overline{u}^p(x_0)\right)- \hat{u}^l(x_0)&=
		H\left(x_0, \overline{u}^l(x_0) - \sum\limits_{p \neq l}  \overline{u}^p(x_0)\right)-{u}^l(x_0)\\
		&=f_{l}(x_0,u^{l}(x_0)).
	\end{align*}
	Now let for some $x_0\in G^o$ we have $u^l(x_0)=0.$ In this case if all $u^p(x_0)=0$ for every $p=1,2,...m,$ then the required inequality follows from the system \eqref{scheme_sys}. Assume there exist some $t\neq l$ such that $u^t(x_0) > 0.$ In this case recalling again the system \eqref{scheme_sys} and that for $s\leq 0$ we extend $H$  function as $H(x,s)=-H(x,-s),$ we obtain:
	\begin{multline*}
		H\left(x_0, \overline{u}^l(x_0) - \sum\limits_{p \neq l}  \overline{u}^p(x_0)\right)- \hat{u}^l(x_0) = H\left(x_0, \overline{u}^l(x_0) - \sum\limits_{p \neq l}  \overline{u}^p(x_0)\right)+ {u}^t(x_0)=\\=
		H\left(x_0, \overline{u}^l(x_0) - \sum\limits_{p \neq l}  \overline{u}^p(x_0)\right)+H\left(x_0, \overline{u}^t(x_0) - \sum\limits_{p \neq t}  \overline{u}^p(x_0)\right) -f_{t}(x_0,u^{t}(x_0))=\\=
		H\left(x_0, \overline{u}^t(x_0) - \sum\limits_{p \neq t}  \overline{u}^p(x_0)\right)-H\left(x_0,  \sum\limits_{p \neq l}  \overline{u}^p(x_0)-\overline{u}^l(x_0) \right) -f_{t}(x_0,u^{t}(x_0)).
	\end{multline*}
	
	It is easy to see that
	\[
	0\leq \overline{u}^t(x_0) - \sum\limits_{p \neq t}  \overline{u}^p(x_0)\leq \sum\limits_{p \neq l}  \overline{u}^p(x_0)-\overline{u}^l(x_0),
	\]
	hence by the non-decreasing property of $H$ one implies
	
	\begin{multline*}
		H\left(x_0, \overline{u}^t(x_0) - \sum\limits_{p \neq t}  \overline{u}^p(x_0)\right)-H\left(x_0,  \sum\limits_{p \neq l}  \overline{u}^p(x_0)-\overline{u}^l(x_0) \right) -f_{t}(x_0,u^{t}(x_0))\leq\\\leq -f_{t}(x_0,u^{t}(x_0))\leq 0\leq f_{l}(x_0,u^{l}(x_0)).
	\end{multline*}
	This completes the proof.
	
\end{proof}

\begin{lemma}\label{lemma1}
	Let the functions $f_{l}(x,s)$ and $H(x,s)$ be nondecreasing with respect to the variable $s$. Assume also that $H$ is a Lipschitz continuous with respect to $s$ and Lipschitz constant is $1,$ i.e. $|H(\cdot,s)-H(\cdot,t)|\leq |s-t|, \forall s,t\in\mathbb{R}^+.$  If any two vectors $(u^{1},u^{2},\dots,u^{m})$ and $(v^{1},v^{2},\dots,v^{m})$ are satisfying the discrete system \eqref{scheme_sys}, then the following equation holds:
	\[
	\max_{x\in V}\left(\hat{u}^l(x)-\hat{v}^l(x)\right)=\max_{\{ x\in V\;:\; u^l(x)\leq v^l(x)\}}\left(\hat{u}^l(x)-\hat{v}^l(x)\right),
	\]
	for all $l=1,2,\dots,m$.
\end{lemma}

\begin{proof}
	We argue by contradiction. Suppose for some $l_0$ we have
	\begin{equation}\label{init_assmp}
		\begin{multlined}
			\hat{u}^{l_0}(x_0)-\hat{v}^{l_0}(x_0)=
			\max_{x\in  V}(\hat{u}^{l_0}(x)-\hat{v}^{l_0}(x))=\\=
			\max_{\{ x\in V\;:\; u^{l_0}(x)> v^{l_0}(x)\}}(\hat{u}^{l_0}(x)-\hat{v}^{l_0}(x))>
			\max_{\{ x\in V\;:\; u^{l_0}(x)\leq v^{l_0}(x)\}}(\hat{u}_\alpha^{l_0}(x)-\hat{v}_\alpha^{l_0}(x)).
		\end{multlined}
	\end{equation}
	
	Then according to Lemma \ref{lemma_disj} the following simple chain of inclusions hold:
	\begin{equation}\label{incl_chain}
		\{u^l(x)> v^l(x)\}\subset\{\hat{u}^l(x)> \hat{v}^l(x)\}\subset\{u^l(x)\geq v^l(x)\}.
	\end{equation}
	We obviously see that $ u^{l_0}(x_0)> v^{l_0}(x_0)\geq 0 $ implies
	$\hat{u}^{l_0}(x_0)>\hat{v}^{l_0}(x_0)$. On the other hand,  the discrete system \eqref{scheme_sys} and Lemma \ref{lemma_ineq} gives us
	$$
	H\left(x_0, \overline{u}^{l_0}(x_0) - \sum\limits_{p \neq l_0}  \overline{u}^p(x_0)\right)- \hat{u}^{l_0}(x_0)= f_{l_0}(x_0,u^{l_0}(x_0)),
	$$
	and	
	$$
	H\left(x_0, \overline{v}^{l_0}(x_0) - \sum\limits_{p \neq l_0}  \overline{v}^p(x_0)\right)- \hat{v}^{l_0}(x_0)\leq f_{l_0}(x_0,v^{l_0}(x_0)).	 	
	$$
	Therefore 	
	\begin{equation*}
		\begin{multlined}
			H\left(x_0, \overline{u}^{l_0}(x_0) - \sum\limits_{p \neq l_0}  \overline{u}^p(x_0)\right)- H\left(x_0, \overline{v}^{l_0}(x_0) - \sum\limits_{p \neq l_0}  \overline{v}^p(x_0)\right)\\ -\left(\hat{u}^{l_0}(x_0)- \hat{v}^{l_0}(x_0)\right)	\geq f_{l_0}(x_0,u^{l_0}(x_0))-f_{l_0}(x_0,v^{l_0}(x_0))\geq 0.
		\end{multlined}
	\end{equation*}
	Thus using non-decreasing and Lipschitz properties of $H$, we conclude
	
	\begin{align}
		0  < \left(\hat{u}^{l_0}(x_0)- \hat{v}^{l_0}(x_0)\right)&\leq
		H\left(x_0, \overline{u}^{l_0}(x_0) - \sum\limits_{p \neq l_0}  \overline{u}^p(x_0)\right)- H\left(x_0, \overline{v}^{l_0}(x_0) - \sum\limits_{p \neq l_0}  \overline{v}^p(x_0)\right)\label{comp-principle}\\&\leq \left(\overline{u}^{l_0}(x_0) - \sum\limits_{p \neq l_0}  \overline{u}^p(x_0)\right) -\left(\overline{v}^{l_0}(x_0) - \sum\limits_{p \neq l_0}  \overline{v}^p(x_0)\right)\nonumber \\&= \frac{1}{\deg(x_0)}\sum_{(x_0,y)\in E}\left(\hat{u}^{l_0}(y)-\hat{u}^{l_0}(y)\right)\nonumber,
	\end{align}
	which implies that $\hat{u}^{l_0}(x_0)-\hat{v}^{l_0}(x_0)=\hat{u}^{l_0}(y)-\hat{v}^{l_0}(y)>0,$ for all $(x_0,y)\in E.$ Due to the chain  \eqref{incl_chain}, we apparently have  ${u}^{l_0}(y)\geq {v}^{l_0}(y)$. According to our assumption \eqref{init_assmp}, the only possibility is ${u}^{l_0}(y)>{v}^{l_0}(y)$  for all $(x_0,y)\in E.$ Now we can proceed the previous steps for all $y\in V$ such that $(x_0,y)\in E,$ and then for each one we will get corresponding neighbours with the same strict inequality and so on. Since the graph $G$ is connected, then  one can always find a path from a given vertex $y$ to the vertex belonging $\partial G.$   Continuing above procedure along this path we will finally approach to the boundary $\partial G$ where as we know ${u}^{l_0}(x)={v}^{l_0}(x)=\phi^{l_0}(x)$ for all $x\in\partial G$. Hence, the strict inequality  fails, which implies that our initial assumption \eqref{init_assmp} is false. Observe that the same arguments can be applied if we interchange the roles of ${u}^{l}(x)$ and ${v}^{l}(x)$. Thus, we also have
	\[
	\max_{G}\left(\hat{v}^l(x)-\hat{u}^l(x)\right)=\max_{\{ v^l(x)\leq u^l(x)\}}\left(\hat{v}^l(x)-\hat{u}^l(x)\right),
	\]
	for every $l=1,2,\dots,m$.

	Particularly, for every fixed $l=1,2.\dots,m$ and $x\in V$ we have	\begin{multline}\label{double_ineq}
		-\max_{\{v^l(x)\leq u^l(x)\}}(\hat{v}^l(x)-\hat{u}^l(x))=	-\max\limits_{x\in V}(\hat{v}^l(x)-\hat{u}^l(x))\leq\\\leq \hat{u}^l(x)-\hat{v}^l(x) \leq 	\max\limits_{x\in V}(\hat{u}^l(x)-\hat{v}^l(x))=\max_{\{u^l(x)\leq v^l(x)\}}(\hat{u}^l(x)-\hat{v}^l(x)).
	\end{multline}
\end{proof}
Thanks to Lemma \ref{lemma1} in the sequel we will use the following notations:
$$
A:=\max_l\;\left(\max\limits_{x\in V}(\hat{u}^l(x)-\hat{v}^l(x))\right)=\max_l\;\left(\max\limits_{\{u^l(x)\leq v^l(x)\}}(\hat{u}^l(x)-\hat{v}^l(x))\right),
$$
and
$$
B:=\max_l\;\left(\max\limits_{x\in V}(\hat{v}^l(x)-\hat{u}^l(x))\right)=\max_l\;\left(\max\limits_{\{v^l(x)\leq u^l(x)\}}(\hat{v}^l(x)-\hat{u}^l(x))\right).
$$

\begin{lemma}\label{lemma2}
	Let the functions $f_{l}(x,s)$ and $H(x,s)$ be nondecreasing with respect to the variable $s$. Assume also that $H$ is a Lipschitz continuous with respect to $s$ and Lipschitz constant is $1,$ i.e. $|H(\cdot,s)-H(\cdot,t)|\leq |s-t|, \forall s,t\in\mathbb{R}^+.$ Let two vectors $(u^{1},u^{2},\dots,u^{m})$ and $(v^{1},v^{2},\dots,v^{m})$  be solving the discrete system \eqref{scheme_sys}. For them  we set $A$ and $B$ as defined above. If $A>0$ and it is attained for some $l_0$, then $A=B>0$  and there exists some $t_0\neq l_0,$ and $y_0\in V,$ such that
	$$
	0<A=\max_{\{u^{l_0}(x)\leq v^{l_0}(x)\}}(\hat{u}^{l_0}(x)-\hat{v}^{l_0}(x))=
	\max_{\{u^{l_0}(x)= v^{l_0}(x)=0\}}(\hat{u}^{l_0}(x)-\hat{v}^{l_0}(x))
	=\hat{v}^{t_0}(y_0)-\hat{u}^{t_0}(y_0).
	$$
	
\end{lemma}
\begin{proof}
	Due to disjointness property of densities $u^l$ and $v^l$ (see Lemma \ref{lemma_disj}) We easily observe that  $(\hat{u}^{l_0}(x)-\hat{v}^{l_0}(x))$ might be positive only on the set $\{u^{l_0}(x)= v^{l_0}(x)=0\}$ (because for the other cases simple checking provides $(\hat{u}^{l_0}(x)-\hat{v}^{l_0}(x))\leq 0)$. Hence,
	$$
	A=\max_{\{u^{l_0}(x)= v^{l_0}(x)=0\}}(\hat{u}^{l_0}(x)-\hat{v}^{l_0}(x)).
	$$
	Using the latter equality, one can prove that $A>0$ implies $B>0$.  Indeed, if we assume that $B\leq 0,$ then according to definition of $B$ we will get that $\hat{v}^{l}(x)\leq\hat{u}^{l}(x)$ for all $l=\overline{1,m}$ and $x\in V$. This obviously yields ${v}^{l}(x)\leq {u}^{l}(x),$ for all $l=\overline{1,m}$ and $x\in V$.  Thus,
	$$
	0<A=\max_{\{u^{l_0}(x)= v^{l_0}(x)=0\}}(\hat{u}^{l_0}(x)-\hat{v}^{l_0}(x))=
	\max_{\{u^{l_0}(x)= v^{l_0}(x)=0\}}\left(
	\sum_{l\neq l_0}(v^{l}(x)-u^{l}(x))\right)\leq 0.
	$$
	This is a contradiction, and therefore $B>0$. It is clear that in a similar way one can prove the converse statement as well. Thus, we clearly see that at the same time either both $A$ and $B$ are non-positive, or they are positive.

	Concerning the equality $A=B,$ it is easy to see that if the maximum $A>0$ is attained at vertex  $y_0\in V,$ then the following holds:
	\begin{equation*}
		0<A=\max_{\{u^{l_0}(x)= v^{l_0}(x)=0\}}(\hat{u}^{l_0}(x)-\hat{v}^{l_0}(x))=
		\hat{u}^{l_0}(y_0)-\hat{v}^{l_0}(y_0)=\sum_{l\neq l_0}(v^l(y_0)-u^l(y_0)).
	\end{equation*}
	The above result implies that $\sum_{l\neq l_0}v^l(y_0)$ is strictly positive,  therefore there exists  $t_0\neq l_0$ such that ${v}^{t_0}(y_0)>0,$ which in turn allows to write ${v}^{t_0}(y_0)= \hat{v}^{t_0}(y_0)$. Hence,
	\begin{equation}\label{A=B}
		0<A={v}^{t_0}(y_0)-\sum_{l\neq l_0}{u}^{l}(y_0)\leq\hat{v}^{t_0}(y_0)-u^{t_0}(y_0)+\sum_{l\neq t_0}u^l(y_0)= \hat{v}^{t_0}(y_0)-\hat{u}^{t_0}(y_0)\leq B.
	\end{equation}
	In the same way we will obtain that $B\leq A,$ and therefore $A=B$. On the other hand, since $A=B,$ then the following obvious equality holds
	$$
	A={v}^{t_0}(y_0)-\sum_{l\neq l_0}{u}^{l}(y_0)= \hat{v}^{t_0}(y_0)-\hat{u}^{t_0}(y_0)=B.
	$$
	This completes the proof.
\end{proof}

\section{Proof of Theorem 1}
\begin{proof}
	
	Let two vectors $(u^{1},u^{2},\dots,u^{m})$ and $(v^{1},v^{2},\dots,v^{m})$  be solving the discrete system \eqref{scheme_sys}. For these vectors we set the definition of $A$ and $B$. Then, we consider two cases $A\leq 0$ and $A>0$. If we assume that $A\leq 0,$ then according to Lemma \ref{lemma2}, we get  $B\leq 0$. But if  $A$ and $B$ are non-positive, then the uniqueness follows. Indeed,  due to  \eqref{double_ineq} we have the following obvious inequalities
	$$
	0\leq -B\leq \hat{u}^l(x)-\hat{v}^l(x)\leq A\leq 0.
	$$
	This provides  for every $l=\overline{1,m}$ and $x\in V$ we have $\hat{u}^l(x)=\hat{v}^l(x),$ which in turn implies $$
	{u}^l(x)={v}^l(x).
	$$
	Now suppose  $A>0$. Our aim is to prove that this case leads to a contradiction. Let the value $A$ is attained for some $l_0\in\overline{1,m},$ then
	due to Lemma \ref{lemma2} there exist $y_0\in V$ and $t_0\neq l_0$ such that:
	\begin{align*}
		0<A=B=&\max_{\{u^{l_0}(x)\leq v^{l_0}(x)\}}(\hat{u}^{l_0}(x)-\hat{v}^{l_0}(x))\\&=\max_{\{u^{l_0}(x)=v^{l_0}(x)=0\}}(\hat{u}^{l_0}(x)-\hat{v}^{l_0}(x))=\hat{v}^{t_0}(y_0)-\hat{u}^{t_0}(y_0).
	\end{align*}
	According to Lemma \ref{lemma_ineq} we have
	$$
	H\left(y_0, \overline{v}^{t_0}(y_0) - \sum\limits_{p \neq t_0}  \overline{v}^p(y_0)\right)- \hat{v}^{t_0}(y_0)= f_{t_0}(y_0,v^{t_0}(y_0)),
	$$
	and	
	$$
	H\left(y_0, \overline{u}^{t_0}(y_0) - \sum\limits_{p \neq t_0}  \overline{u}^p(y_0)\right)- \hat{u}^{t_0}(y_0)\leq f_{t_0}(y_0,u^{t_0}(y_0)).	 	
	$$
	
	Recalling  that $f_{l}(x,s)$ and $H(x,s)$ are non-decreasing with respect to the variable $s$, also $H$ is a Lipschitz continuous w.r.t. $s,$ along with the fact that $\hat{v}^{t_0}(y_0)>\hat{u}^{t_0}(y_0)$ implies ${v}^{t_0}(y_0)\geq {u}^{t_0}(y_0),$ we can repeat the same steps as in \eqref{comp-principle} to obtain
	$$
	\left(\hat{v}^{t_0}(y_0)- \hat{u}^{t_0}(y_0)\right)\leq \frac{1}{\deg(y_0)}\sum_{(y_0,z)\in E}\left(\hat{v}^{t_0}(z)-\hat{u}^{t_0}(z)\right).
	$$
	This  implies  $A=\hat{v}^{t_0}(y_0)-\hat{u}^{t_0}(y_0)=\hat{v}^{t_0}(z)-\hat{u}^{t_0}(z)>0$ for all $(y_0,z)\in E$.  The chain \eqref{incl_chain} provides that  for all $(y_0,z)\in E,$ we have ${v}^{t_0}(z)\geq{u}^{t_0}(z)$. Since a graph $G$ is connected, then one can always find a path from $y_0$ to some boundary vertex $w\in\partial G.$ Assume the vertices along this path are $y_0;y_1;\dots;y_{k-1};y_k=w.$ Hence, for every $0 \leq j\leq k-1,$ we have $(y_j,y_{j+1})\in E,$ i.e. every vertex  $y_{j+1}$ is a closest neighbor for $y_{j}$ and $y_{j+2}.$

	According to the above arguments for the neighbor vertex $y_1\in V$ we proceed as follows: If  ${v}^{t_0}(y_1)>{u}^{t_0}(y_1),$ then obviously
	$$
	\left(\hat{v}^{t_0}(y_1)- \hat{u}^{t_0}(y_1)\right)\leq \frac{1}{\deg(y_1)}\sum_{(y_1,z)\in E}\left(\hat{v}^{t_0}(z)-\hat{u}^{t_0}(z)\right).
	$$
	This, as we saw a few lines above, leads to  $A=\hat{v}^{t_0}(y_1)-\hat{u}^{t_0}(y_1)=\hat{v}^{t_0}(z)-\hat{u}^{t_0}(z)>0$ for all $(y_1,z)\in E$.  In particular,
	$A = \hat{v}^{t_0}(y_2)-\hat{u}^{t_0}(y_2)> 0.$
	
	If  ${v}^{t_0}(y_1)={u}^{t_0}(y_1),$  then due to $\hat{v}^{t_0}(y_1)-\hat{u}^{t_0}(y_1)=A=B>0,$ there exists some $\lambda_0\neq t_0,$ such that
	$$
	0<A=\hat{v}^{t_0}(y_1)-\hat{u}^{t_0}(y_1)=
	\sum_{l\neq t_0}\left({u}^{l}(y_1)-{v}^{l}(y_1) \right)={u}^{\lambda_0}(y_1)-\sum_{l\neq t_0}{v}^{l}(y_1).
	$$
	Note that ${u}^{\lambda_0}(y_1)>0$ implies ${u}^{l}(y_1)=0$ for all $l\neq \lambda_0,$ and particularly  ${v}^{t_0}(y_1)={u}^{t_0}(y_1)=0.$
	Following the definition of $A,$ we get
	$$
	A={u}^{\lambda_0}(y_1)-\sum_{l\neq t_0}{v}^{l}(y_1)\geq \hat{u}^{\lambda_0}(y_1)-\hat{v}^{\lambda_0}(y_1),
	$$
	which in turn gives $2\sum\limits_{l\neq \lambda_0}{v}^{l}(y_1)\leq 0,$ and therefore
	${v}^{l}(y_1)=0$ for all $l\neq \lambda_0$. Hence
	$$
	A={u}^{\lambda_0}(y_1)-\sum_{l\neq t_0}{v}^{l}(y_1)= \hat{u}^{\lambda_0}(y_1)-\hat{v}^{\lambda_0}(y_1),
	$$
	This suggests us to apply the same approach as above to arrive at
	$$
	\left(\hat{v}^{\lambda_0}(y_1)- \hat{u}^{\lambda_0}(y_1)\right)\leq \frac{1}{\deg(y_1)}\sum_{(y_1,z)\in E}\left(\hat{v}^{\lambda_0}(z)-\hat{u}^{\lambda_0}(z)\right),
	$$
	which leads to
	$A=\hat{u}^{\lambda_0}(y_1)-\hat{v}^{\lambda_0}(y_1)=\hat{u}^{\lambda_0}(z)-\hat{v}^{\lambda_0}(z)>0,$ for all $(y_1,z)\in E$. In particular,
	$A = \hat{u}^{\lambda_0}(y_2)-\hat{v}^{\lambda_0}(y_2)> 0.$
	Thus, combining two cases we observe that for $y_2\in V$ there exist an index $1\leq l_{y_2}\leq m$ (in our case $l_{y_2}=t_0$ or $l_{y_2}=\lambda_0$) such that
	\begin{equation}\label{contradiction}
		\mbox{either}\;\;	\hat{u}^{l_{y_2}}(y_2)-\hat{v}^{l_{y_2}}(y_2)=A, \;\;\mbox{or}\;\;
		\hat{u}^{l_{y_2}}(y_2)-\hat{v}^{l_{y_2}}(y_2)=-A.
	\end{equation}
	It is not hard to understand that the same procedure can be repeated for a vertex $y_2$ instead of $y_1$ and come to the same conclusion \eqref{contradiction} for $y_3\in V$ and some index $l_{y_3}$ and so on. This allows to claim that for every $y_j\in V$ along the path $(y_0,\dots,y_k)$  there exist some $l_{y_j}$ such that
	\[
	|\hat{u}^{l_{y_j}}(y_j)-\hat{v}^{l_{y_j}}(y_j)|=A>0.
	\]
	But this means that above equality holds also for $y_k=w\in \partial G,$ which will lead to a contradiction, because for every $z\in\partial G,$ and $l=\overline{1,m}$ one has $\hat{u}^{l}(z)-\hat{v}^{l}(z)=0$.  This completes the proof.
\end{proof}
Thus,  the system \eqref{scheme_sys} has a  unique solution. Hence, one can think of constructing an iterative method which will converge to the unique solution. To this aim, in the forthcoming work the authors are going to develop the efficient iterative methods to solve the systems of type \eqref{scheme_sys}.

\section*{Acknowledgments}
F. Bozorgnia was  supported by   the Portuguese National Science Foundation through FCT fellowships   and    by European Unions Horizon 2020 research and innovation program under Marie Skłodowska-Curie grant agreement No. 777826 (NoMADS).

\bibliographystyle{plain}%
\bibliography{library}

\end{document}